\documentclass[12pt,reqno]{amsart}

\usepackage{fullpage}   
\usepackage{url}

\newtheorem{lemma}{Lemma}
\newtheorem{theorem}[lemma]{Theorem}
\newtheorem{proposition}[lemma]{Proposition}
\newtheorem{problem}{Problem}
\newtheorem*{conjecture}{Conjecture}
\theoremstyle{remark}
\newtheorem*{acknowledgment}{Acknowledgment}

\numberwithin{equation}{section}

\newcommand{\Ker}{\operatorname{Ker}}
\renewcommand{\Im}{\operatorname{Im}}

\title{On a Problem of M. Kambites Regarding Abundant Semigroups}

\author{Jo\~{a}o Ara\'{u}jo}
\author{Michael Kinyon}

\address[Ara\'{u}jo]
{Universidade Aberta
and
Centro de \'{A}lgebra \\
Universidade de Lisboa \\
1649-003 Lisboa \\ Portugal}
\email{\url{jaraujo@ptmat.fc.ul.pt}}

\address[Kinyon]{Department of Mathematics \\
University of Denver \\ 2360 S Gaylord St \\ Denver, Colorado 80208 USA}
\email{\url{mkinyon@du.edu}}

\begin{document}

\begin{abstract}
A semigroup is \emph{regular} if it contains at least one idempotent in each $\mathcal{R}$-class and in each $\mathcal{L}$-class. A regular semigroup is \emph{inverse} if it satisfies either of
the following equivalent conditions: (i) there is a unique idempotent in each $\mathcal{R}$-class and in each $\mathcal{L}$-class, or (ii) the idempotents commute.

Analogously, a semigroup is \emph{abundant} if it contains at least one idempotent in each $\mathcal{R}^*$-class and in each $\mathcal{L}^*$-class. An abundant semigroup is \emph{adequate} if its idempotents commute. In adequate semigroups, there is a unique idempotent in each $\mathcal{R}^*$ and $\mathcal{L}^*$-class. M. Kambites raised the question of the converse: in a finite abundant semigroup such that there is a unique idempotent in each $\mathcal{R}^*$ and $\mathcal{L}^*$-class, must the idempotents commute? In this note we provide a negative answer to this question.
\end{abstract}

\maketitle

\section{Introduction}

For a set $X$, denote by $T(X)$ the monoid of all total transformations of $X$ to itself.
For $t\in T(X)$, let $\Im(t)= \{ xt \mid x\in X\}$ and $\Ker(t)=\{(x,y) \mid xt=yt\}$
denote, respectively, the image of $X$ under $t$ and the kernel of $t$.

Let $S$ be a semigroup. We denote the right regular representation and
the left regular antirepresentation of $S$ by
$\rho:S\rightarrow T(S^1) ; s \mapsto \rho_s$ and $\lambda: S\rightarrow T(S^1) ; s \mapsto \lambda_s$,
respectively, where $(x)\rho_s = xs$ and $\lambda_s(x) = sx$ for all $x\in S^1$.

Now we can define the usual Green's equivalence relations on $S$ as follows:
for $s,t\in S$,
\[
s\mathcal{L}t \Leftrightarrow \Im (\rho_s)=\Im(\rho_t)
\qquad\text{and}\qquad
s\mathcal{R}t \Leftrightarrow \Im (\lambda_s)=\Im(\lambda_t)\,.
\]
Similarly we can define two more equivalence relations on $S$ as follows: for $s,t\in S$,
\[
s\mathcal{L}^* t \Leftrightarrow \Ker (\lambda_s)=\Ker(\lambda_t)
\qquad\text{and}\qquad
s\mathcal{R}^* t \Leftrightarrow \Ker (\rho_s)=\Ker(\rho_t)\,.
\]
More simply put, $s\mathcal{L}^* t$ if and only if, for all $x,y\in S^1$,
$sx=sy\Leftrightarrow tx=ty$, and similarly, $s\mathcal{R}^* t$ if and only if,
for all $x,y\in S^1$, $xs=ys\Leftrightarrow xt=yt$.
It is not difficult to show that $\mathcal{L}\subseteq \mathcal{L}^*$ and
$\mathcal{R}\subseteq \mathcal{R}^*$. In addition, it is well known that
$s\mathcal{L}^* t$ if and only if $s\mathcal{L} t$ in some semigroup containing $S$
as a subsemigroup, and a similar characterization holds for $\mathcal{R}^*$.

A semigroup is \emph{regular} if there is at least one idempotent in every
$\mathcal{L}$-class and in every $\mathcal{R}$-class. A semigroup is
\emph{abundant} if there is at least one idempotent in every
$\mathcal{L}^*$-class and in every $\mathcal{R}^*$-class \cite{fountain2}.

Within regular semigroups, the notion of \emph{inverse} semigroup is characterized
by either of the following equivalent properties: (i) the idempotents commute, or
(ii) there exists a unique idempotent
in every $\mathcal{L}$-class and in every $\mathcal{R}$-class.

A generalization of (i) is as follows: a semigroup is \emph{adequate} if it is
abundant and the idempotents commute \cite{fountain1}. Adequate semigroups satisfy the starred
analog of (ii), that is, in an adequate semigroup,
there exists a unique idempotent in every $\mathcal{L}^*$-class and
in every $\mathcal{R}^*$-class. Fountain constructed an infinite example of an abundant, nonadequate semigroup with unique idempotents in each $\mathcal{L}^*$-class and
in each $\mathcal{R}^*$-class (\cite{fountain1}, Example 1.4).

At the second meeting of the North British Semigroups and Applications Network
(16-17 April 2009) \cite{MK}, Mark Kambites asked a natural question:

\medskip

\noindent\emph{Let} $S$ \emph{be a finite abundant semigroup such that
there exists a unique idempotent in every} $\mathcal{L}^*$\emph{-class and
in every} $\mathcal{R}^*$\emph{-class. Do the idempotents of} $S$ \emph{commute?}

\medskip

We will say that a semigroup is \emph{amiable} if each $\mathcal{R}^*$-class and each $\mathcal{L}^*$-class
contains a unique idempotent. Every amiable semigroup is obviously abundant. Kambites'
question can be rephrased as: \emph{Is every finite amiable semigroup adequate?}

Kambites' question attracted the attention of a number of semigroup theorists
attending the NBSAN meeting. The aim of this note is to answer the question in the negative.

\section{The Smallest Example}

In this section we construct the smallest example of an amiable semigroup with a pair
of noncommuting idempotents, and show further that, up to isomorphism, it is the unique
such semigroup of its size.

A key observation used in the arguments that follow is that for idempotents $a,b$ in
a semigroup, $a \mathcal{L}^* b$ if and only if $a\mathcal{L} b$, and similarly for
$\mathcal{R}^*$ and $\mathcal{R}$.

\begin{lemma}
\label{lem:atleast4}
Let $S$ be an amiable semigroup and assume $a,b\in S$ are noncommuting idempotents.
Let $c = ab$ and $d = ba$. Then the elements $a,b,c,d$ are distinct. Further,
$ac = cb = c$, $bd = da = d$, $ad = ca = cd = aba \neq a,b$ and $bc = db = dc = bab \neq a,b$.
\end{lemma}

\begin{proof}
Suppose, for instance, $c = a$. Then $dd = baba = bca = ba = d$, that is, $d$ is an idempotent.
Since $da = d$ and $ad = a$, we have $a \mathcal{L} d$. Since $S$ is amiable,
$a = d$, which contradicts the assumption that $c\neq d$. Thus $c\neq a$,
and by a similar argument, $d\neq b$. Exchanging the roles of $a$ and $b$, we also have
$d\neq a$ and $c \neq b$. This establishes the first claim.

If $aba = a$, then $abab = ab$ so that $ab$ is an idempotent. Since $a(ab) = ab$ and
$(ab)a = a$, we have $ab\mathcal{R}a$. Since $S$ is amiable, $ab = a$,
a contradiction. If $aba = b$, then $ab = aaba = aba = b$, another contradiction.
That $bab \neq a,b$ follows from exchanging the roles of $a$ and $b$. The remaining
equalities are clear.
\end{proof}

It follows from Lemma \ref{lem:atleast4} that an amiable semigroup
which is not adequate must have order at least $4$. A partial multiplication table
looks like this:
\begin{table}[htb]
\[
\begin{tabular}{c|cccc}
    & $a$ & $b$ & $c$ & $d$ \\ \hline
  $a$ & $a$ & $c$ & $c$ & $e$ \\
  $b$ & $d$ & $b$ & $f$ & $d$ \\
  $c$ & $e$ & $c$ & $g$ & $e$ \\
  $d$ & $d$ & $f$ & $f$ & $h$
\end{tabular}
\]
\end{table}

\noindent where $e = ad = ca = cd$, $f = bc = db = dc$ and $g = cc = eb = af$ are not necessarily
distinct from $a,b,c,d$ or from each other.

\begin{lemma}
\label{lem:basic}
Let $S$ be an amiable semigroup and assume $a,b\in S$ are noncommuting idempotents.
Then $aba = ab$ if and only if $bab = ab$. When these conditions hold, $a$ and $b$
generate a subsemigroup of $S$ of order $4$.
\end{lemma}

\begin{proof}
Suppose first that $aba = ab$. Then $abab = ab$ and $(bab)(bab) = bab$, that is,
$ab$ and $bab$ are idempotents. Since $(ab)(bab) = ab$ and $(bab)(ab) = bab$, we have
$ab \mathcal{L} bab$. Since $S$ is amiable, $bab = ab$.

Conversely, if $bab = ab$, then again $ab$ is an idempotent and also
$(aba)(aba) = aba$, that is, $aba$ is an idempotent. Since $(aba)(ab) = ab$
and $(ab)(aba) = aba$, we have $ab \mathcal{R} aba$. Since $S$ is amiable, $aba = ab$.

Now assume the conditions $aba = bab = ab$ hold.
Set $c = ab$ and $d = ba$. By Lemma \ref{lem:atleast4} and the discussion that
follows, $ad = ca = cd = bc = db = dc = cc = c$, and so
$M = \{a,b,c,d\}$ is a subsemigroup of $S$ with multiplication table given in
Table 1.
\end{proof}

\begin{table}
\[
\begin{tabular}{c|cccc}
  $M$ & $a$ & $b$ & $c$ & $d$ \\ \hline
  $a$ & $a$ & $c$ & $c$ & $c$ \\
  $b$ & $d$ & $b$ & $c$ & $d$ \\
  $c$ & $c$ & $c$ & $c$ & $c$ \\
  $d$ & $d$ & $c$ & $c$ & $c$
\end{tabular}
\]
\caption{The smallest amiable semigroup which is not adequate.}
\end{table}

In Lemma \ref{lem:basic}, the magma $M$ of Table 1 is a semigroup because $M$ is closed
under the multiplication of the semigroup $S$. Now we show that $M$ is a semigroup in its
own right.

\begin{lemma}
\label{lem:referee}
The magma $M$ defined by Table 1 is a semigroup.
\end{lemma}

\begin{proof}
For $s\in M$, define $\tau_s : M\to M$ by
$(t)\tau_s = ts$ for all $t\in M$. We prove associativity by showing that $\tau_s \tau_t = \tau_{st}$ for all $s,t\in M$. Order the elements of $M$ by $a<b<c<d$. Then the mapping $\tau_s$ can be represented by the (transposed) column corresponding to $s$ in Table 1. Thus we write
\[
\tau_a = [\ a\ d\ c\ d\ ], \qquad \tau_b = [\ c\ b\ c\ c\ ], \qquad
\tau_c = [\ c\ c\ c\ c\ ], \qquad \tau_d = [\ c\ d\ c\ c\ ]\,.
\]
The entry of the transformation $\tau_s$ indexed by $t\in M$ is $(t)\tau_s$.
Since $c$ is a zero of $M$, we have $\tau_c \tau_s = \tau_c = \tau_{cs}$ and similarly
$\tau_s \tau_c = \tau_{sc}$ for all $s\in M$. We verify the other nine cases by direct calculation,
recalling that composition is from left to right:
\begin{alignat*}{3}
\tau_a \tau_a &= [\ a\ d\ c\ d\ ]\ [\ a\ d\ c\ d\ ] &&= [\ a\ d\ c\ d\ ] &&= \tau_a = \tau_{aa} \\
\tau_a \tau_b &= [\ a\ d\ c\ d\ ]\ [\ c\ b\ c\ c\ ] &&= [\ c\ c\ c\ c\ ] &&= \tau_c = \tau_{ab} \\
\tau_a \tau_d &= [\ a\ d\ c\ d\ ]\ [\ c\ d\ c\ c\ ] &&= [\ c\ c\ c\ c\ ] &&= \tau_c = \tau_{ad} \\
\tau_b \tau_a &= [\ c\ b\ c\ c\ ]\ [\ a\ d\ c\ d\ ] &&= [\ c\ d\ c\ c\ ] &&= \tau_d = \tau_{ba} \\
\tau_b \tau_b &= [\ c\ b\ c\ c\ ]\ [\ c\ b\ c\ c\ ] &&= [\ c\ b\ c\ c\ ] &&= \tau_b = \tau_{bb} \\
\tau_b \tau_d &= [\ c\ b\ c\ c\ ]\ [\ c\ d\ c\ c\ ] &&= [\ c\ d\ c\ c\ ] &&= \tau_d = \tau_{bd} \\
\tau_d \tau_a &= [\ c\ d\ c\ c\ ]\ [\ a\ d\ c\ d\ ] &&= [\ c\ d\ c\ c\ ] &&= \tau_d = \tau_{da} \\
\tau_d \tau_b &= [\ c\ d\ c\ c\ ]\ [\ c\ b\ c\ c\ ] &&= [\ c\ c\ c\ c\ ] &&= \tau_c = \tau_{db} \\
\tau_d \tau_d &= [\ c\ d\ c\ c\ ]\ [\ c\ d\ c\ c\ ] &&= [\ c\ c\ c\ c\ ] &&= \tau_c = \tau_{dd}\,. \tag*{\qedhere}
\end{alignat*}
\end{proof}

\begin{theorem}
\label{thm:order4}
Up to isomorphism, there is exactly one amiable semigroup of order $4$ which is not adequate.
\end{theorem}

\begin{proof}
By Lemma \ref{lem:referee}, $M$ is a semigroup. The $\mathcal{L}^*$-classes are
$\{a,d\}$, $\{b\}$ and $\{c\}$. The $\mathcal{R}^*$-classes are
$\{b,d\}$, $\{a\}$ and $\{c\}$. Hence $M$ is amiable, but not adequate since $ab = c\neq d=ba$.

Now let $S$ be an amiable semigroup of order $4$ and assume that $a,b\in S$
are noncommuting idempotents. Set $c = ab$ and $d = ba$.
By Lemma \ref{lem:atleast4}, $S = \{a,b,c,d\}$ and $aba \neq a,b$.
Thus $aba = c$ or $aba = d$. The desired isomorphism of $S$ with $M$ then follows from
Lemma \ref{lem:basic}, directly for the case $aba=c$, and with the roles of
$a$ and $b$ reversed for the case $aba=d$.
\end{proof}

\section{Larger Examples}

A computer search for amiable semigroups which are not adequate revealed that up to order $37$,
every such semigroup contains a copy of the order $4$ example $M$. Thus we offer the following.

\begin{conjecture}
Let $S$ be a finite amiable semigroup which is not adequate. Then $S$ contains a subsemigroup isomorphic to $M$.
\end{conjecture}

As some corroborating evidence for the conjecture, we now prove a special case.

\begin{proposition}
\label{prp:order5}
Let $S$ be an amiable semigroup of order $5$ which is not adequate. Then $S$ contains a subsemigroup
isomorphic to $M$.
\end{proposition}

\begin{proof}
Assume $S = \{a,b,c,d,e\}$ where $a$ and $b$ are noncommuting idempotents, $ab =c$ and $ba = d$.
(Here we use Lemma \ref{lem:atleast4} once again to guarantee that $a,b,c,d$ are distinct.)
Suppose that $S$ does not contain a copy of $M$. By Lemmas \ref{lem:atleast4} and \ref{lem:basic},
$aba \neq a,b,c,d$, and so $aba = e$. By the same argument with the roles of $a$ and $b$ reversed,
$bab = e$. Then $ae = ea = e$, $be = eb = e$, $ce = ed = ababa = aea = e$,
$de = ec = babab = beb = e$ and $ee = ababa = e$. Thus we have the multiplication table given in Table 2.
\begin{table}[htb]
\[
\begin{tabular}{c|ccccc}
   & $a$ & $b$ & $c$ & $d$ & $e$  \\ \hline
  $a$ & $a$ & $c$ & $c$ & $e$ & $e$ \\
  $b$ & $d$ & $b$ & $e$ & $d$ & $e$ \\
  $c$ & $e$ & $c$ & $e$ & $e$ & $e$ \\
  $d$ & $d$ & $e$ & $e$ & $e$ & $e$ \\
  $e$ & $e$ & $e$ & $e$ & $e$ & $e$
\end{tabular}
\]
\caption{}
\end{table}
However, we see immediately by inspection of the table that $d$ does not have any
idempotent in its $\mathcal{L}^*$-class. This contradicts the abundance of $S$.
\end{proof}

Note that there do exist amiable semigroups with noncommuting idempotents which themselves do not
generate a copy of $M$. The smallest order where this occurs is $8$. Table 3 gives an example.

\begin{table}[htb]
\[
\begin{tabular}{c|cccccccc}
    & $a$ & $b$ & $c$ & $d$ & $e$ & $f$ & $g$ & $h$ \\ \hline
  $a$ & $a$ & $c$ & $c$ & $c$ & $f$ & $f$ & $c$ & $f$ \\
  $b$ & $d$ & $b$ & $c$ & $d$ & $b$ & $g$ & $g$ & $c$ \\
  $c$ & $c$ & $c$ & $c$ & $c$ & $c$ & $c$ & $c$ & $c$ \\
  $d$ & $d$ & $c$ & $c$ & $c$ & $g$ & $g$ & $c$ & $g$ \\
  $e$ & $d$ & $b$ & $c$ & $d$ & $e$ & $g$ & $g$ & $h$ \\
  $f$ & $c$ & $c$ & $c$ & $c$ & $f$ & $c$ & $c$ & $f$ \\
  $g$ & $c$ & $c$ & $c$ & $c$ & $g$ & $c$ & $c$ & $g$ \\
  $h$ & $c$ & $c$ & $c$ & $c$ & $h$ & $c$ & $c$ & $h$
\end{tabular}
\]
\caption{}
\end{table}

\noindent Here $a$ and $e$ are noncommuting idempotents with $aea \neq ae, ea$ and
$eae \neq ae, ea$.
Note that this semigroup contains two copies of $M$, namely $\{a,b,c,d\}$ and
$\{h,a,c,f\}$. Indeed, every example known
to us contains \emph{some} pair of noncommuting idempotents which generates a copy of $M$.

\section{Idempotent and Pseudozero Inflations}

In this section, we will discuss constructions of new amiable semigroups
from existing ones by adjoining a single element. Of course, two obvious ways
of doing this are by adjoining an identity element or by adjoining a zero.
Here we discuss two other constructions.

Let $S$ be a semigroup and $e^2=e\in S$. The \emph{idempotent inflation of} $S$ \emph{induced by} $e$ is a magma $(U,\circ)$ such that $U=S\cup \{\epsilon\}$ (where $\epsilon$ is a symbol not in $S$), and for $x,y \in U$,
\[
x\circ y=\begin{cases}
     xy, & \text{if}\ x,y\in S \\
     ey, & \text{if}\ x=\epsilon, y\in S \\
     xe, & \text{if}\ x\in S, y=\epsilon \\
     \epsilon, & \text{if}\ x=y=\epsilon\,.
\end{cases}
\]
The idempotent inflation of $S$ induced by $e^2=e\in S$ will be denoted by $S[e]$.

\begin{lemma}
For a semigroup $S$ with idempotent $e\in S$, the idempotent inflation $S[e]$ is a semigroup.
\end{lemma}

\begin{proof}
We show that for $x,y,z\in S$, $x\circ (y\circ z)=(x\circ y)\circ z$.
For $x,y\in S$ and $z=\epsilon$,
$(x\circ y)\circ \epsilon = (xy)e = x(ye) = x\circ (y\circ \epsilon)$.
Similarly, $(x\circ \epsilon)\circ y = x\circ (\epsilon\circ y)$ and
$\epsilon\circ (y\circ z) = (\epsilon\circ y)\circ z$. Finally, for $x\in S$,
$x\circ (\epsilon \circ \epsilon) = x\circ \epsilon = xe = (xe)e
= (xe)\circ \epsilon = (x\circ \epsilon)\circ \epsilon$.
Similarly, $(\epsilon \circ x)\circ \epsilon = \epsilon\circ (x\circ \epsilon)$
and $(\epsilon \circ \epsilon)\circ x = \epsilon\circ (\epsilon \circ x)$.
\end{proof}

Note that $S[e]$ is an inflation of $S$ in the usual sense (see \cite[ex.10, p.98]{clifford}) because the map $\phi: S[e] \to S$
given by $x\phi = x$ for $x\in S$ and $\epsilon\phi = e$ is an idempotent homomorphism onto $S$.

For a semigroup $S$ with idempotent $e\in S$ and idempotent inflation $U = S[e]$,
we will write $\mathcal{R}_S^*$ and $\mathcal{R}_U^*$ for the $\mathcal{R}^*$
relations in $S$ and $U$, respectively, and similarly for the $\mathcal{L}^*$ relations.

\begin{lemma}
\label{lem:stars}
Let $S$ be a semigroup with idempotent $e\in S$ and let $U = S[e]$ be the idempotent
inflation of $S$ induced by $e$. Then
(i) the $\mathcal{R}_U^*$-classes are $\{\epsilon\}$ and the $\mathcal{R}_S^*$-classes, and
(ii) the $\mathcal{L}_U^*$-classes are $\{\epsilon\}$ and the $\mathcal{L}_S^*$-classes.
\end{lemma}

\begin{proof}
Firstly suppose $s\mathcal{R}_U^* \epsilon$
for some $s\in S$. Since $1\circ \epsilon = \epsilon = \epsilon\circ \epsilon$, we have
$1\circ s = \epsilon\circ s = es = e\circ s$. But then
$1\circ \epsilon = e\circ \epsilon$, that is, $\epsilon = e$, which is a contradiction.
Therefore the $\mathcal{R}_U^*$-class of $\epsilon$ is $\{\epsilon\}$. The remaining
$\mathcal{R}_U^*$-classes are contained in $S$.

Next, if $s\mathcal{R}_U^* t$ for $s,t\in S$, then $s\mathcal{R}_T t$ for some
oversemigroup $T$ of $U$. Since $T$ is also an oversemigroup of $S$, $s\mathcal{R}_S^* t$.
Thus each $\mathcal{R}_U^*$-class is contained in an
$\mathcal{R}_S^*$-class. What remains is to show the reverse inclusion. Thus suppose
$s\mathcal{R}_S^* t$
for some $s,t\in S$. For $x\in S^1$, $\epsilon\circ s = x\circ s$ if and only if $es = xs$
if and only if $et = xt$ if and only if $\epsilon\circ t = x\circ t$. Therefore, for
all $x,y\in U^1$, $x\circ s=y\circ s\Leftrightarrow x\circ t=y\circ t$. This shows
$s\mathcal{R}_U^* t$ as desired.

The corresponding results for the $\mathcal{L}^*$ relations follow by symmetry.
\end{proof}

Immediately from Lemma \ref{lem:stars}, we have the following.

\begin{theorem}
\label{thm:inflation-preserve}
Let $S$ be a semigroup, and let $e\in S$ be an idempotent.
Then $S$ is abundant, amiable or adequate if and only if $S[e]$ has the same corresponding property.
\end{theorem}

Let $S$ be a semigroup with a zero $0\in S$. The \emph{pseudozero inflation} of $S$
is a magma $(U,\circ)$ where $U = S\cup \{\bar{0}\}$ (where $\bar{0}$ is a symbol
not in $S$) and for $x,y\in U$,
\[
x\circ y=\begin{cases}
     xy, & \text{if}\ x,y\in S \\
     \bar{0}, & \text{if}\ x=\bar{0}, y\in S\ \text{or}\ x\in S, y=\bar{0} \\
     0, & \text{if}\ x=y=\bar{0}\,.
\end{cases}
\]
The pseudozero inflation of $S$, where $0\in S$ is a zero, will be denoted by $S^{\bar{0}}$.

\begin{lemma}
For a semigroup $S$ with zero $0\in S$, the pseudozero inflation $S^{\bar{0}}$ is a semigroup.
\end{lemma}

\begin{proof}
We show that for $x,y,z\in S$, $x\circ (y\circ z)=(x\circ y)\circ z$.
For $x,y\in S$ and $z=\bar{0}$,
$(x\circ y)\circ \bar{0} = \bar{0} = x\circ (y\circ \bar{0})$.
Similarly, $(x\circ \bar{0})\circ y = x\circ (\bar{0}\circ y)$ and
$\bar{0}\circ (y\circ z) = (\bar{0}\circ y)\circ z$. For $x\in S$,
$x\circ (\bar{0} \circ \bar{0}) = x\circ 0 = 0 =
\bar{0}\circ \bar{0} = (x\circ \bar{0})\circ \bar{0}$. Similarly,
$(\bar{0} \circ \bar{0})\circ x = \bar{0}\circ (\bar{0} \circ x)$ and
 $(\bar{0} \circ x)\circ \bar{0} = \bar{0}\circ (x\circ \bar{0})$.
\end{proof}

As in the idempotent case,
$S^{\bar{0}}$ is an inflation of $S$ in the usual sense because the map
$\phi: S^{\bar{0}} \to S$
given by $x\phi = x$ for $x\in S$ and $\bar{0}\phi = 0$ is an idempotent
homomorphism onto $S$.

For a semigroup $S$ with zero $0\in S$ and pseudozero inflation $U = S^{\bar{0}}$,
we will again write $\mathcal{L}_S^*$ and $\mathcal{L}_U^*$ for the $\mathcal{L}^*$
relations in $S$ and $U$, respectively, and similarly for the $\mathcal{R}^*$ relations.

\begin{lemma}
\label{lem:pseudo-stars}
Let $S$ be a semigroup with zero $0\in S$ and let $U = S^{\bar{0}}$ be the pseudozero
inflation of $S$. Then
(i) the $\mathcal{R}_U^*$-classes are $\{0,\bar{0}\}$ and the $\mathcal{R}_S^*$-classes contained in $S\backslash \{0\}$, and
(ii) the $\mathcal{L}_U^*$-classes are $\{0,\bar{0}\}$ and the $\mathcal{L}_S^*$-classes contained in $S\backslash \{0\}$.
\end{lemma}

\begin{proof}
We have $x\circ\bar{0} = y\circ\bar{0}$ if and only if either $x, y\in S^1$ or $x=y=\bar{0}$.
For $s\in S$, we have $x\circ s = y\circ s$ for all $x,y\in S^1$ if and only if $s = 0$.
It follows that $\{0,\bar{0}\}$ is an $\mathcal{R}_U^*$-class, and so the remaining
$\mathcal{R}_U^*$-classes are contained in $S\backslash \{0\}$.

As in the proof of Lemma \ref{lem:stars}, if $s\mathcal{R}_U^* t$ for some $s,t\in S\backslash \{0\}$,
then $s\mathcal{R}_S^* t$ since any oversemigroup of $U$ is an oversemigroup of $S$.
Conversely, suppose $s\mathcal{R}_S^* t$ for $s,t\in S$. If for $x,y\in U^1$, we have
$x\circ s = y\circ s$, then either $x,y\in S^1$ or $x=y=\bar{0}$. In either case,
$x\circ t = y\circ t$. Similarly, $x\circ t = y\circ t$ implies $x\circ s = y\circ s$
for all $x,y\in U^1$. Therefore $s\mathcal{R}_U^* t$.

The corresponding results for the $\mathcal{L}^*$ relations follow by symmetry.
\end{proof}

\begin{theorem}
\label{thm:pseudozero-preserve}
Let $S$ be a semigroup with zero $0\in S$.
Then $S$ is abundant, amiable or adequate if and only if $S^{\bar{0}}$ has the same corresponding property.
\end{theorem}

If an amiable semigroup is not adequate, then neither its pseudozero inflation
nor any of its idempotent inflations will be adequate, and so we have two ways
of generating new amiable semigroups which are not adequate from existing ones.
But these one point inflations certainly do not tell the whole story, even for
small orders. We verified the following result computationally.

\begin{theorem}
\label{thm:order5}
The following seven semigroups are the unique (up to isomorphism) amiable
semigroups of order $5$ which are not adequate:
(i) $M^1$ ($M$ with an identity element adjoined),
(ii) $M^0$ ($M$ with a zero adjoined),
(iii) $M^{\bar{c}}$ (the pseudozero inflation of $M$),
(iv) $M[a]$, $M[b]$, $M[c]$ (the three idempotent inflations of $M$) and
(v) the semigroup given by the following table:
\begin{table}[htb]
\[
\begin{tabular}{c|ccccc}
      & a & b & c & d & e\\
\hline
    a & a & c & c & c & c \\
    b & d & b & c & d & e \\
    c & c & c & c & c & c \\
    d & d & c & c & c & c \\
    e & e & c & c & c & c\\
\end{tabular}
\]
\end{table}
\end{theorem}

\section{Problems}

Besides our main conjecture in {\S}3, we also suggest the following.

For each $x$ in an amiable semigroup, there is a unique idempotent $x_{\ell}$ in the $\mathcal{L}^*$-class of $x$
and a unique idempotent $x_r$ in the $\mathcal{R}^*$-class of $x$. One can view such semigroups as algebras of type
$\langle 2,1,1\rangle$ where the binary operation is the semigroup multiplication and the unary operations are
$x\mapsto x_{\ell}$ and $x\mapsto x_r$. Such algebras form a quasivariety axiomatized by, for instance, the
eight quasi-identities

\begin{center}
\begin{tabular}{ccc}
$x_{\ell} x_{\ell} = x_{\ell}$ & \qquad & $x_r x_r = x_r$ \\
$x x_{\ell} = x$ & & $x_r x = x$ \\
$xy = xz \Rightarrow x_{\ell}y = y_{\ell}z$ & & $yx = zx \Rightarrow yx_r = zx_r$
\end{tabular}

\begin{tabular}{c}
$(\ xx = x\ \&\ yy = y\ \&\ xy = x\ \&\ yx = y\ ) \Rightarrow x = y$ \\
$(\ xx = x\ \&\ yy = y\ \&\ xy = y\ \&\ yx = x\ ) \Rightarrow x = y$.
\end{tabular}
\end{center}

\noindent Fountain's infinite example \cite{fountain1} and our finite examples show that the quasivariety of amiable semigroups properly contains the quasivariety of adequate semigroups.
Kambites determined the free objects in the latter quasivariety \cite{MK2}, and thus the
following is quite natural.

\begin{problem}
Determine the free objects in the quasivariety of amiable semigroups.
\end{problem}

As we saw in Theorem \ref{thm:order5}, there are amiable semigroups with noncommuting idempotents which are not given by idempotent or pseudozero inflations or by
adjoining an identity element or zero element to a given amiable semigroup. Consider, for the moment, constructions such as direct products and amiable subsemigroups to be ``trivial''.

\begin{problem}
Find other nontrivial ways of building new amiable semigroups from existing ones.
\end{problem}

\begin{acknowledgment}
We are pleased to acknowledge the assistance of
the finite model builder \textsc{Mace4} and the automated deduction tool \textsc{Prover9},
both developed by
McCune \cite{McCune}. We initially found small examples using
\textsc{Mace4} and then reverse engineered them into the examples
presented here. \textsc{Prover9} was helpful in working out the proofs in {\S}2.

The first author was partially supported by FCT and FEDER,
Project POCTI-ISFL-1-143 of Centro de Algebra da Universidade de Lisboa,
and by FCT and PIDDAC through the project PTDC/MAT/69514/2006.
\end{acknowledgment}


\begin{thebibliography}{99}

\bibitem{clifford} A. H. Clifford and G. B. Preston,
The algebraic theory of semigroups I,
Mathematical Surveys, N. 7, American Mathematical Society, Providence, R.I. 1961.

\bibitem{fountain1} J. Fountain, Adequate semigroups,
\textit{Proc. Edinburgh Math. Soc.} (2) \textbf{22} (1979), no. 2, 113--125.

\bibitem{fountain2} J. Fountain, Abundant semigroups,
\textit{Proc. London Math. Soc.} (3) \textbf{44} (1982), no. 1, 103--129.

\bibitem{MK} M. Kambites,
\url{http://www.maths.manchester.ac.uk/~mkambites/events/nbsan/}

\bibitem{MK2} M. Kambites,
Free adequate semigroups,
\textit{J. Aust. Math. Soc.}, to appear,
\url{arXiv:0902.0297}.

\bibitem{McCune} W. McCune, \emph{Prover9 and Mace4},
version 2009-11A, (\url{http://www.cs.unm.edu/~mccune/prover9/}).

\end{thebibliography}
\end{document}